\documentclass[a4paper,11pt]{amsart}

\usepackage{amsfonts}
\usepackage{amsmath}
\usepackage{amsthm}
\usepackage{amssymb}

\theoremstyle{definition} \newtheorem{defn}{Definition}[section]
\theoremstyle{plain} \newtheorem{thm}[defn]{Theorem}
\theoremstyle{plain} \newtheorem{propn}[defn]{Proposition}
\theoremstyle{plain} \newtheorem{lemma}[defn]{Lemma}
\theoremstyle{plain} 
\theoremstyle{plain} 

\theoremstyle{definition}  
\theoremstyle{definition}  
\theoremstyle{remark}  
\theoremstyle{remark}  
\theoremstyle{remark}  

\theoremstyle{plain} \newtheorem*{thm*}{Theorem}
\theoremstyle{plain} \newtheorem*{cor*}{Corollary}

\newcommand {\R} {\mathbb{R}}
\newcommand {\Q} {\mathbb{Q}}

\newcommand {\Z} {\mathbb{Z}}

\newcommand {\N} {\mathbb{N}}

\newcommand {\A} {\mathcal{A}}

\newcommand{\C}{\mathbb{C}}
\renewcommand {\epsilon}{\varepsilon}
\newcommand {\abs}[1] {\left\vert #1 \right\vert}

\author{Gareth Jones}
\address{gareth.jones-3@manchester.ac.uk}
\address{School of Mathematics, University of Manchester, Oxford Road, Manchester, M13 9PL, UK.}
\author{Shi Qiu}
\address{shi.qiu-3@postgrad.manchester.ac.uk}
\address{School of Mathematics, University of Manchester, Oxford Road, Manchester, M13 9PL, UK.}
\title[Integer-valued definable functions]{Integer-valued definable functions in $\R_{\text{an},\text{exp}}$}
\begin{document}
\begin{sffamily}
\maketitle
\begin{abstract} We give two variations on a result of Wilkie's (\cite{Wilkie}) on unary functions definable in $\R_{\text{an},\exp}$ that take integer values at positive integers. Provided that the function grows slower (in a suitable sense) than the function $2^x$, Wilkie showed that is must be eventually equal to a polynomial. We show the same conclusion under a stronger growth condition but only assuming that the function takes values sufficiently close to a integers at positive integers. In a different variation we show that it suffices to assume that the function takes integer values on a sufficiently dense subset of the positive integers (for instance the primes), again under a stronger growth bound than that in Wilkie's result. 
\end{abstract}
\section{Introduction}
%{\bf still to do: comments on constants. Comments on several variables? Fix references. Abstract. Check!! (Especially: second theorem. Add comment concerning epsilon in Wilkie's result. Fix T' to kT' in second result. }

In \cite{JTW}, Thomas, Wilkie and the first author studied functions $f:(0,\infty)\to\R$ definable in certain o-minimal structures under the assumption that $f$ is integer-valued, that is $f(n)\in \Z$ for all positive integer $n$. They showed that, under a rather strong growth bound, such a function must be a polynomial. This gives a weak real analogue of a classical theorem of Polya on integer-valued entire functions. Wilkie \cite{Wilkie} substantially improved the one-variable result of \cite{JTW}, proving that such an $f$ must be a polynomial provided that it satisfies a growth bound that is close to optimal. Wilkie's result also applies to a larger o-minimal structure than those considered in \cite{JTW}.

Here we consider a similar problem in which the function $f$ is no longer supposed integer-valued, but only assumed to be such that $f(n)$ is close to an integer for positive integers $n$. Throughout this paper, by definable, we mean definable in the structure $\R_{\text{an},\exp}$. This structure is o-minimal by work of van den Dries and Miller \cite{vdDMiller}. We prove the following.
\begin{thm}\label{close_to_integers}
 There is a $\delta>0$ with the following property. Suppose that $f:[0,\infty)\to \R$ is definable and analytic, and that there exists $c_0>0$ such that for all positive integers $n$ there is an integer $m_n$ such that
\[
\abs{f(n)-m_n}<c_0 e^{-3n}.
\]
If there are $c_1>0$ and $\delta'<\delta$ such that $|f(x)|<c_1e^{\delta' x}$ then there is a polynomial $Q$ such that
\[
Q(n)=m_n
\]
for all sufficiently large $n$.
\end{thm}

We also prove a result in which our function $f$ is only assumed to be integer-valued on some sufficiently dense subset of the nonnegative integers. For this we fix $\A\subseteq \N$ such that there is a positive real  $\lambda$ such that for all sufficiently large $T$ we have 
\[
\frac{T}{(\log T)^\lambda} \ll \# \A \cap [0,T] \ll \frac{T}{(\log T)^\lambda}.
\]
With such an $\A$ fixed, we prove the following.
\begin{thm} \label{Primes} Suppose that $f:[0,\infty)\to \R$ is definable and analytic, and such that $f(n)$ is an integer for $n\in\A$. If there exist $\alpha>0$ and $c_1>0$ such that
\[
|f(x)|<c_1 \exp \left(\frac{x}{(\log x)^{2\lambda+2+\alpha}}\right)
\]
then $f$ is a polynomial.
\end{thm}

So for instance, if $|f(x)|<c \exp\left(\frac{x}{(\log x)^5}\right)$ and $f$ is integer-valued on the primes, then $f$ is a polynomial.

Before discussing the proofs of these results, we briefly discuss Wilkie's proof. Wilkie first shows that definable unary functions whose growth is at most exponential can be approximated by a function which admits an analytic continuation (as a complex function) to a half-plane. The diophantine part of the proof then follows Polya's method, as adapted by Langley \cite{Langley} to functions on a half-plane. This seems to need the function to be take integer values at all positive integers, and doesn't seem to adapt to a set $\A$ as in the second theorem above. And it is difficult to see how this method could be used to prove Theorem \ref{close_to_integers}. The method of \cite{JTW}, which relies on a counting theorem in \cite{JT} gives nothing in the context of Theorem \ref{close_to_integers}. Instead, we adapt Waldschmidt's proof, via transcendence methods, of a weak form of Polya's Theorem \cite{Waldschmidt}. See also Chapter 9 of Masser's recent book \cite{Masser}. Waldschmidt's proof was adapted by Hirata \cite{Hirata} to show that an entire $f$ that is exponentially close to integers at positive integers, and doesn't grow too quickly must be a polynomial. And a more precise version of this result was given by  Ito and Hirata-Kohno  \cite{Ito_HirataKohno}. This fails in the o-minimal setting (consider say $e^{- a x}$ for large $a$, or even worse $\exp(- \exp x)$). But the method, applied not to the function itself but to an approximating function with a large continuation, provided by Wilkie's work, goes through and gives Theorem \ref{close_to_integers}. The second Theorem is proved by similar methods, and further exploits o-minimality, in that to show that an analytic function definable in an o-minimal structure is identically zero, it is enough to show that it has infinitely many zeros. In the setting of our second Theorem, the method of \cite{JTW}, applying the counting theorem for curves in \cite{JT} would give that $f$ is a polynomial provided that it is definable in the structure $\R_{\exp}$ and satisfies the much stronger growth bound $|f(x)|<e^{x^\epsilon}$, eventually, for all positive $\epsilon$. 

It will be clear from the proofs that a similar result can be obtained, for instance, assuming that $f$ is close to integers (in the sense of Theorem \ref{close_to_integers}) on a sequence $\A$ as in Theorem \ref{Primes}. Or we could suppose that the sequence $\A$ in Theorem \ref{Primes} had more points, and get a corresponding relaxation in the growth bound. Various results of this nature will appear in the thesis of the second author. Finally, the main point of \cite{JTW} was to consider functions of several variables that take integer values at tuples with integer coordinates. Again, considerations of this kind will appear in the thesis of the second author.

\subsubsection*{Acknowledgements} We would like to thank Alex Wilkie for very helpful discussions concerning \cite{Wilkie}. We are also most grateful to the referee for their comments and suggestions.

\section{Preliminaries}
We begin with various estimates that we shall use repeatedly. First, a well-known estimate for binomial coefficients (see for instance (9.9) on page 104 of \cite{Masser}). 
\begin{lemma}\label{binomial}
Suppose that $z\in \C$ and that $i$ is a nonnegative integer with $i\le L$. Then
\[
\abs{\binom{z}{i} }\le e^L\left( \frac{\abs{z}+L}{L}\right)^L.
\]
\end{lemma}

The following lemma is easy. 
\begin{lemma}\label{polydiff}
Suppose that $P=\sum_{i=0}^L\sum_{j=0}^M p_{i,j}\binom{X}{i}Y^j$ is a polynomial with complex coefficients. Then for $a>0 $ and $b,b'\ge 1$ we have
\[
\abs{P(a,b)-P(a,b')} \le (L+1)(M+1)^2\max \{ |p_{i,j}|\} \max \left\{ \abs{\binom ai }: i \le L  \right\} b^Mb'^M|b-b'|.
\]
\end{lemma}

We will use the following estimate which is a special case of Lemma 3 of \cite{Tijdeman}. 
\begin{lemma} Suppose that $a_1,\ldots,a_N$ are pairwise distinct integers. Then
\[
\prod_{i=2}^N\abs{a_1-a_i} \ge \frac{(N-1)!}{2^{N-1}}.
\]
\end{lemma}

In place of the usual formula for the number of zeros of a complex function, we use the following, a special case of Lemma 6 of \cite{Ito_HirataKohno}.
\begin{propn}\label{jensen} Suppose that $\phi$ is analytic on $|z|\le R$, and that $a_1,\ldots, a_N$ are nonzero complex numbers of modulus less than $R$. Then 
\begin{equation}\label{jensenestimate}
|\phi(0)| \le |\phi|_R \prod_{i=1}^N \frac{ |a_i|}{R} + \sum_{i=1}^N \left( |\phi(a_i)| \prod_{j=1}^N \frac{ |R^2 - a_i\overline{a_j}|}{R^2}\prod_{k=1,k\ne n}^N \frac{|a_k|}{|a_k-a_i|}\right).
\end{equation}
\end{propn}

\section{Functions with values close to integers}
For our proofs we will use Wilkie's results on approximate continuations. The result we need in this section is as follows. 
\begin{thm}[Wilkie] Suppose that $f:\R \to \R$ is definable and suppose that there exist $c_1>0$ and $\delta>0$  such that $|f(x)|<c_1 e^{\delta x}$ for all large $x$. Given $D>0$ and $\epsilon>0$ there is an $a \in \R$ and an analytic $g: \{ z : \text{Re}(z) >a\} \to \C $ such that 
\begin{enumerate}
\item $|g(x) - f(x) | < e^{ -D x}$ for all $x >a$,
\item there exists $c_2>0$ such that $|g(z)| <c_2 e^{(\delta+\epsilon) |z|}$ for all $z$ such that $\text{Re}(z) >a$. 
\end{enumerate}
\end{thm}
\begin{proof} Apply Corollary 4.8 in \cite{Wilkie} to get definable functions $f_1,\ldots, f_l :(a,\infty)\to \R$, pairwise distinct reals $s_1,\ldots,s_l$ and a positive $a$ such that 
\begin{equation}\label{close1}
\abs{ f(x) - \sum_{i=1}^l f_i(x) \exp(s_i x)} < \exp (- D x),
\end{equation}
for all $x>a$. Moreover, in the notation of \cite{Wilkie}, the functions $f_i$ are all in $\mathcal{R}_{\text{subexp}}$. It then follows from Theorem 4.2 in \cite{Wilkie} that, perhaps after increasing $a$, these functions continue to the half-plane $\{ z: \text{Re}(z) >a \}$. And by Lemma 5.3 in \cite{Wilkie}, increasing $a$ if necessary, these continuations are such that
\[
|f_i(z)| \le \exp (\epsilon |z|)
\]
for any $z$ in the half-plane with large modulus, and $i=1,\ldots, l$. By the growth condition on $f$ and \eqref{close1}, we have that $s_i \le \delta$ for each $i$. So we can take $g$ to be the continuation of $\sum_{i=1}^l f_i(x) \exp(s_i x)$.

\end{proof}

\begin{thm} There is a $\delta>0$ with the following property. Suppose that $f:[0,\infty)\to \R$ is definable and analytic, and that there exists $c_0>0$ such that for all positive integers $n$ there is an integer $m_n$ such that
\[
\abs{f(n)-m_n}<c_0 e^{-3n}.
\]
If there are $c_1>0$ and $\delta'<\delta$ such that $|f(x)|<c_1e^{\delta' x}$ then there is a polynomial $Q$ such that
\[
Q(n)=m_n
\]
for all sufficiently large $n$.
\end{thm}

\begin{proof}
Rather than the assumption in the statement, we start by assuming that the integers $m_n$ are such that 
\[
\abs{f(n)-m_n}<c_0 e^{-Dn}
\]
for some positive $D$, and show later that the choice $D=3$ is sufficient. Note that we can assume without loss of generality that $f$ is positive. Suppose that $ f(x) \le c_1e^{\delta' x}$, for some positive $\delta'<\delta<1$, where we fix $\delta$ later.   By Wilkie's Theorem above, there is a function $g$ analytic in some right halfplane, and such that $|g(x)-f(x)|<e^{-Dx}$ for large $x$. Moreover, we have $|g(z)|\le c_2 e^{\delta |z|}$. We will assume that $c_2\ge 2 c_1$. Translating, we can assume that the halfpane is $\{ z: \text{Re} z \ge 0\}$, and that for all $x\ge 0$ we have
\begin{equation}\label{fgclose}
|g(x)-f(x)| < c_3 e^{-D x}.
\end{equation}

We let $M,L$ be large integers, to be determined later. We suppose that $M+1<L$ and that $L$ is odd. Set $T=(L+1)(M+1)$. We construct a nonzero polynomial
\[
P(X,Y)= \sum_{i=0}^L\sum_{j=0}^M p_{i,j} \binom{X}{i} Y^j
\]
such that
\begin{equation}\label{constraints}
P(n,m_n)=0
\end{equation}
for $n\in [T/2,T)\cap \Z$. So we have $T$ unknowns and $T/2$ equations. To apply Siegel's lemma, we first compute an upper bound on
\begin{displaymath}
\abs{\binom{n}{i}m_n^j}
\end{displaymath}
for $n\in [T/2, T]\cap \Z, i\le L$ and $j\le M$. By Lemma \ref{binomial} we have
\[
\abs{\binom{n}{i}} \le e^L(M+4)^L.
\]
And
\[
m_n^j \le (2f(n))^j \le c_2^M e^{\delta M T}.
\]
So Siegel's Lemma (see \cite[8.3]{Masser}, for example) gives solutions $p_{i,j}$ to \eqref{constraints}, not all zero, with
\begin{equation}\label{normP}
\abs{p_{i,j}} \le (L+1)(M+1) e^L(M+4)^L c_2^M e^{\delta MT}.
\end{equation}
We now show that $P(X,Y)$ is not too big at either $(n,f(n))$ or $(n,g(n))$. For $n\in  [T/2,T]\cap \Z$ we have, by assumption, $|f(n)-m_n |\le c_0 e^{-DT/2}$. So by Lemma \ref{polydiff}, we have, for $n$ in the same range
\[
|P(n,f(n))-P(n,m_n)| \le T^3 c_2^{3M}e^{2L}(M+4)^{2L}e^{3\delta M T-DT/2}.
\]
In particular, the right hand side here is an upper bound for $|P(n,f(n))|$ for $n\in [T/2,T)\cap \Z$. Similarly, for $n \in [T/2, T]\cap \Z$ we have
\[
|P(n,f(n))-P(n,g(n))| \le  T^3 c_2^{3M}e^{2L}(M+4)^{2L}e^{3\delta M T-DT/2}.
\]

Combining these, we have
\begin{equation}\label{tobeat}
%\begin{aligned}
|P(T,m_{T})| \le |P(T,g(T))|+ 2 T^3 c_2^{3M}e^{2L}(M+4)^{2L}e^{3\delta M T-DT/2}.
%\end{aligned}
\end{equation}
And for $n\in [T/2,T)\cap \Z$ we have
\begin{equation}\label{supgn}
|P(n,g(n))| \le 2 T^3 c_2^{3M}e^{2L}(M+4)^{2L}e^{3\delta M T-DT/2}.
\end{equation}

With the aim of showing that $P(T ,m_{T  })=0$, we now consider the function $\phi(z)=P(z+T ,g(z+T ))$, analytic on $|z|\le T $ and show, using Proposition \ref{jensen}, that it is small at the origin. For $|z|\le 2T $ and $i\le L$ we have
\[
\abs{\binom{z}{i}} \le e^{2L}(M+4)^L
\]
(using Lemma \ref{binomial}). So, from the definition of $\phi$, the bounds \eqref{normP} and the growth bound in Wilkie's Theorem we have
\begin{equation}\label{supphi}
|\phi|_T\le T^2 c_2^{2M}e^{3L}(M+4)^{2L} e^{3\delta M T}.
\end{equation}
Let $\A_T=[T/2,T)\cap \Z$, and for $n\in \A_T$ let $a_n=n-T $. Then $|a_n|\le  T/2$ and
\begin{eqnarray*}
|\phi|_T \prod_{n \in \A_T} \frac{|a_n|}{T} & \le |\phi|_T \left(\frac{1}{2}\right)^{\frac{T}{2}}.\\
\end{eqnarray*}
 So by \eqref{supphi} we have
\[
|\phi|_T \prod_{n \in \A_T}  \frac{|a_n|}{T} \le T^2 c_2^{2M} \left( \frac{e^{3}(M+4)^2}{2^{\frac{M+1}{2}}}\right)^{L+1} e^{3\delta M T}
\]
%Taking $M$ sufficiently large (for instance $M=37$ works) we have
%\[
 %\left( e^{3}(M+4)^2\frac{11^{\frac{M+1}{2}}}{20^{\frac{M+1}{2}}}\right) < \frac{1}{2}.
%\]
%We then take $\delta>0$ so small that
%\[
%e^{3\delta M(M+1)}<2.
%\]
%With these choices, the above shows that
%\[
%|\phi|_T \prod_{n \in \A_T} \frac{|a_n|}{T}
%\]
%decays exponentially with $L$.

We now consider the second summand in the estimate \eqref{jensenestimate} for $|\phi(0)|$. First, we have
\[
\prod_{k \in \A_T\setminus \{ n\}} \frac{1}{|a_k-a_n|} \le \frac{2^{\frac{T}{2}-1}}{\left( \frac{T}{2}-1\right)!}
\]
for any $n\in \A_T$. And so
\[
\prod_{k \in \A_T\setminus \{ n\}} \frac{|a_k|}{|a_k-a_n|} \le  \frac{  \left( \frac{T}{2}\right)^{\frac{T}{2}-1} 2^{\frac{T}{2}-1}}{\left( \frac{T}{2}-1\right)!}
\]
Estimating the factorial and simplifying, this is at most
\[
c_4(2e)^{\frac{T}{2}-1}
\]
for some positive constant $c_4$. Since
\[
\abs{\frac{T^2-a_na_k}{T^2} }<1
\]
we have the following upper bound (using \eqref{supgn}) for the second summand in \eqref{jensenestimate}.
\[
2c_4 T^4 c_2^{3M}e^{2L}(M+4)^{2L}e^{3\delta M T-DT/2}(2e)^{\frac{T}{2}-1}.
\]

So, by \eqref{tobeat}, and Proposition \ref{jensen} we have
\begin{align*}
|P(T ,m_{T })| \le& T^2 c_2^{2M} \left( \frac{e^{3}(M+4)^2}{2^{\frac{M+1}{2}}}\right)^{L+1}e^{3\delta M T} + c_4 T^4 c_2^{3M}e^{2L}(M+4)^{2L}e^{3\delta M T)-DT/2}(2e)^{\frac{T}{2}} \\
\le & c_4T^4 c_2^{3M}\left( \frac{e^{3}(M+4)^2}{2^{\frac{M+1}{2}}}\right)^{L+1} e^{3\delta MT}  \left( 1 + \exp \left( - \frac{DT}{2} + T\log 2 +\frac{T}{2}\right)\right)
\end{align*}

% 4 c_0 c_3c' (L+1)^3(M+1)^4 2^{2M} c_2^{3M} \\
%& \cdot  \left( \left( e^3(M+4) \frac{11^{\frac{M+1}{2}}}{20^{\frac{M+1}{2}}}\right)^{L+1} e^{\delta M (3T+1)}
%+e^{2L}(M+4)^{2L} e^{3\delta M(T+1)-DT/2}(2e)^{T/2}\right) \\
%&\cdot  \left( \left( e^3(M+4) \frac{11^{\frac{M+1}{2}}}{20^{\frac{M+1}{2}}}\right)^{L+1} e^{\delta M (3T+1)} +e^{2L}(M+4)^{2L} e^{3\delta M(T+1)-DT/2}(2e)^{T/2}\right)\\
%\le &  4 c_0 c_3c' (L+1)^3(M+1)^4 2^{2M} c_2^{3M} \\
%\le &  4 c_0 c_3c' (L+1)^3(M+1)^4 2^{2M} c_2^{3M}
%& \left( \left(   e^3(M+4) \frac{11^{\frac{M+1}{2}}}{20^{\frac{M+1}{2}}}e^{3\delta M ((M+1)+1)}\right)^{L+1}\left( 1 + e^{(2 \log 2 +1 - D)T/2}\right)\right)
Now, if we fix $M$ sufficiently large then
\[
\left( \frac{e^{3}(M+4)^2}{2^{\frac{M+1}{2}}}\right)< \frac{1}{2}.
\]
We then take $\delta>0$ so small that
\[
e^{3\delta M(M+1)}<2.
\]
These choices ensure that
$$
 \left( \frac{e^{3}(M+4)^2}{2^{\frac{M+1}{2}}}e^{3\delta M (M+1)}\right)^{L+1}
$$
decays exponentially as $L$ increases. Then taking
\[
D> 2\log 2 +1
\]
(e.g. $D=3$ as claimed in the statement of the Thereom) we will have
\[
|P(T ,m_{T }) |< 1
\]
provided that $L$ is large enough. As $P(T ,m_{T })$ is an integer, it must be zero.

Now inductively suppose that $P(n,m_n)=0$ for all $n \in [T/2,T')\cap \Z$, for some integer $T'>T$. We write $T'=(T+1)(M'+1)$ for some $M'\in \Q$. Using Lemma \ref{binomial}, if $|z|\le T' $ and $i\le L$ then
\begin{equation}\label{binomT'+1}
\abs{ \binom{z}{i} } \le e^L (2M'+2)^L.
\end{equation}
And if $|z| \le 2T' $ and $i \le L$ then
\begin{equation}\label{binom2T'+1}
\abs{ \binom{z}{i} } \le e^L (3M'+4)^L.
\end{equation}
Using \eqref{binomT'+1} and Lemma \ref{polydiff}, for $n \in [T'/2,T']$ we have
\begin{equation}
%\begin{aligned}
|P(n,f(n))-P(n,m_n)| \le T^3 c_2^{2M}e^{2L}(M+4)^L(2 M'+2)^L e^{3 \delta MT'-DT'/2}.
%\end{aligned}
\end{equation}
Similarly, for $n$ in the same range,
\begin{equation}
|P(n,f(n)) - P(n,g(n))| \le T^3 c_2^{2M}e^{2L}(M+4)^L(2 M'+2)^L e^{3 \delta MT'-DT'/2}.
\end{equation}
These two inequalities hold in particular for $n \in [T'/2, T')\cap \Z$ and here we also have $P(n,m_n)=0$, so for $n$ in this range we get
\begin{equation}\label{modP(n(gn))}
|P(n,g(n))| \le 2 T^3 c_2^{2M}e^{2L}(M+4)^L(2 M'+2)^L e^{3 \delta MT'-DT'/2}.
\end{equation}
Finally, for $n=T' $ we have
\begin{equation}\label{stilltobeat}
%\begin{aligned}
|P(T',m_{T' })| \le |P(T' , g(T' ))| +2 T^3 c_2^{2M}e^{2L}(M+4)^L(2 M'+2)^L e^{3 \delta MT'-DT'/2}.
%\end{aligned}
\end{equation}

As before, we now aim to show that $P(T',g(T'))$ is small, using Proposition \ref{jensen}, and thus show that $P(T' , m_{T' })=0$. To this end let
\[
\psi(z) = P(z+T', g(z+T' )),
\]
analytic on $|z|\le T' $. Let $\A_{T'}= [T'/2 , T')\cap \Z$ so that $ \frac{T'}{2}-1\le \# \A_{T'} \le \frac{T'}{2}$. And for $n\in \A_{T'} $ let $a_n= n-T'$, so that $|a_n| \le \frac{T'}{2}$. Using the definition of $\psi$, together with \eqref{normP} and \eqref{binom2T'+1} and the growth bounds on $g$, we have
\[
|\psi|_{T'} \le T^2 c_2^{2M}e^{2L}(M+4)^L(3M'+4)^L e^{3\delta MT'}.
\]
From this we see that
\begin{equation}\label{G1}
|\psi|_{T'} \prod_{n\in \A_{T'}} \frac{|a_n|}{T'}\le  T^2 c_2^{2M} \left( \frac{e^2(M+4)(3M'+4)}{2^{\frac{M'+1}{2}}}\right)^{L+1} e^{3\delta MT'}.
\end{equation}
To estimate the second summand in \eqref{jensenestimate} we first note that as in the base case, for $n \in \A_{T'}$, we have
\[
\prod_{k \in \A_{T'} \setminus \{ n\}} \frac{|a_k|}{|a_k-a_n|} \le  c_5 (2 e)^{\frac{T'}{2}-1}
\]
for some positive constant $c_5$. And as before, for $n,i \in \A_{T'}$ we have
\[
\abs{\frac{T'^2-a_ia_n}{T'^2}} <1.
\]
So, with \eqref{modP(n(gn))} we have
\begin{equation}\label{G2}
\begin{split}
 \sum_{i\in \A_{T'}}&\left(  |\psi(a_i)| \prod_{j\in \A_{T'}}  \frac{ |T'^2 - a_ia_j|}{T'^2} \prod_{k\in \A_{T'},k\ne n}  \frac{|a_k|}{|a_k-a_n|}\right) \le \\ & c_5T'T^3c_2^{2M}e^{2L}(M+4)^L(2M'+2)^Le^{3\delta MT'-D T'/2}(2e)^{T'/2 -1}.
\end{split}
\end{equation}
So by Proposition \ref{jensen} we get an upper bound for $|P(T',g(T' ))|$, given by the sum of the right hand sides of \eqref{G1} and \eqref{G2}, and then by \eqref{stilltobeat} we have
\[
\begin{split}
|P(T' , m_{T' })| \le c_5T' T^3c_2^{2M}&\left( \frac{e^2(M+4)(3M'+4)}{2^{\frac{M'+1}{2}}}\right)^{L+1} \cdot \\ & e^{3\delta M  T'  }\left(1+ e^{(-D +2 \log 2 +1)T'/2 }\right)
% &\le   2 (L+1)^2(M+1)^22^Mc_1^Mc_2^M\left( e^2(M+4)(3M'+3)  \left( \frac{11}{20}\right)^{\frac{M'+1}{2}}\right)^{L+1}\\
%&\cdot e^{\delta M T+\delta M (2T'+1)} + 4c_0c_3 c''(L+1)^2 (M+1)^3 2^{2M} c_1^{3M}T' e^{2L}\\ 
%&\cdot(M+4)^L (2M'+2)^L e^{\delta MT + 2 \delta M (T'+1) - DT'/2} (2 e)^{\frac{T}{2}-2}\\
%& +2c_0c_3 (L+1)^2 (M+1)^3 2^{2M} c_1^{3M} e^{2L}\\ 
%&\cdot(M+4)^L (2M'+2)^L e^{\delta MT + 2 \delta M (T'+1) - DT'/2} \\
%&\le 4 c_0c_3c''(L+1)^3(M+1)^3 2^{2M}c_1^{3M}c_2^M (M'+1) \\
 % &\cdot \left( e^2(M+4)(3M'+3)\left( \frac{11}{20}\right)^{\frac{M'+1}{2}}\right)^{L+1}\\
  % &\cdot e^{\delta M T +2 \delta M (T'+1)} \left( 1 + e^{ (2\log 2 +1 -D)T'/2   } \right).
%\end{aligned}
\end{split}
\]
With our earlier choices of $M, \delta$ and $D$, and $L$ fixed sufficiently large, this tends to $0$ as $M' \ge M$ tends to infinity. So $P(T' ,m_{T' })=0$.

Hence, by induction, we have $P(n,m_n)=0$ for all $n$.

Now, there are finitely many analytic algebraic functions $\theta_1,\ldots,\theta_k$ defined on some interval $(a,\infty)$ such that if $P(x,y)=0$ and $x>a$ then $y=\theta_i(x)$ for some $i$. We assume that the functions $\theta_i$ are distinct. For $n>a$ we have $m_n = \theta_i(n)$ for some $i$. We show that $i$ here is independent of $n$, perhaps after increasing $a$. Suppose on the contrary that $i,j\le k$ are not equal and such that there are infinitely many $n>a$ with $m_n=\theta_i(n)$ and infinitely many $n'>a$ with $m_{n'}=\theta_j(n')$. Since for all large $n$ we have
\[
|f(n)-m_n| < c_0 e^{-3 n}
\]
we have, by o-minimality,
\[
|\theta_i(x)-\theta_j(x)| < 2 c_0 e^{-3 x}
\]
for all large $x$. But $\theta_i$ and $\theta_j$ are algebraic, so this cannot happen unless they are equal. Hence there is some $i$ such that for all sufficiently large $n$ we have $m_n =\theta_i(n)$. Since $\theta$ is an algebraic function taking integer values at all large integers, $\theta_i$ must be a polynomial (for instance by the Theorem on page 131 of \cite{Serre}), and the proof is complete.
\end{proof}

\section{Functions taking integer values on a reasonably dense sequence.}
For this section we require another result on approximate continuations.
\begin{thm}[Wilkie] Suppose that $f:\R \to \R$ is definable and suppose that there exist positive $N,\alpha$ and $c_1$  such that $|f(x)|<c_1 \exp \left(\frac{x}{(\log x)^{N+\alpha}}\right)$ for all large $x$. Then there exist $\eta>0,a \in \R$ and an analytic $g: \{ z : \text{Re}(z) >a\} \to \C $ such that 
\begin{enumerate}
\item $|g(x) - f(x) | < e^{ -\eta x}$ for all $x >a$,
\item there exists $c_2>0$ such that $|g(z)| <c_2 \exp \left(\frac{x}{(\log x)^N}\right)$ for all $z$ such that $\text{Re}(z) >a$. 
\end{enumerate}
\end{thm}
\begin{proof} By Theorems 4.1 and 4.2 in \cite{Wilkie}, there exist   $a\in \R$ and an analytic $g:\{ z : \text{Re}(z) >a\} \to \C $ such that $|g-f|$ is infinitesimal with respect to the valuation ring $\mathcal{F}_{\text{subexp}}$ (in the notation of \cite{Wilkie}). We may suppose that $|g-f|$ is positive. So $|g-f|$ is less than all elements of the decreasing sequence
\[
\exp\left(-\frac{x}{\log x}\right), \exp\left(-\frac{x}{\log \log x}\right),\ldots,\exp\left(-\frac{x}{\log\cdots \log x}\right),\ldots.
\]
It then follows from the fact that $\R_{\text{an},\exp}$ is exponentially bounded (see Proposition 9.2 of \cite{vdDMiller}) that there is a positive $\eta$ such that the first condition in the theorem holds.

It remains to check the growth condition. This follows using Theorem 4.2 of \cite{Wilkie}, following the argument of Lemma 5.3 of \cite{Wilkie} but with the $\phi$ there replaced by 
$$
\frac{(\log x)^N \log g(x)}{x}.$$
The extra $\alpha$ in the growth bound ensures that this $\phi$ tends to $0$, so that Wilkie's argument works. 
\end{proof}

With this in hand, we prove the second main result. Recall that we fix a set $\A$ of positive integers for which there exist positive real $\lambda$ such that for sufficiently large $T$ we have
\[
\frac {T}{(\log T)^\lambda} \ll \A \cap [0,T] \ll \frac{T}{(\log T)^\lambda}.
\]
It follows that there exist positive reals $a,b$ and $\epsilon \in (0,1)$ such that 
\[
a \frac {T}{(\log T)^\lambda} \le \A\cap [\epsilon T,T] \le b \frac {T}{(\log T)^\lambda}
\]
for large $T$. 

The result is as follows. 
\begin{thm} Suppose that $f:[0,\infty)\to \R$ is definable and analytic, and such that $f(n)$ is an integer for $n\in\A$. If there exist $\alpha>0$ and $c_1>0$ such that
\[
|f(x)|<c_1 \exp\left(\frac{x}{(\log x)^{2\lambda+2+\alpha}}\right)
\]
then $f$ is a polynomial.
\end{thm}
\begin{proof} To prove this we proceed as before and start by applying Wilkie's theorem above and translating to obtain a positive $\eta$ and a $g$, analytic on the right half-plane $\{ z : \text{Re}(z)\ge 0\}$, such that 
\begin{equation}\label{primes:g_growth}
|g(z)| \le c_2 \exp \left(\frac{|z|}{(\log |z|)^{2\lambda+2}}\right)
\end{equation}
and
\begin{equation}\label{primes:g_close_to_f}
|f(x)-g(x)| \le c_3 e^{-\eta x}
\end{equation}
for $x\ge 0$, where $c_2$ and $c_3$ are positive and we assume $c_2\ge c_1$.  Below we set $\delta(x)=\frac{x}{(\log x)^{2\lambda+2}}$.

We fix large integers $L$ and $M$ and aim to construct a nonzero polynomial 
\[
P(X,Y)= \sum_{i=0}^L\sum_{j=0}^M p_{i,j} \binom{X}{i}Y^j
\]
with integer coefficients such that 
\begin{equation}\label{primes:constraint}
P(n,f(n))=0
\end{equation}
for all $n\in \A\cap [\epsilon T,T]$, where $T=(L+1)(M+1)$. Aiming to use Siegel's Lemma to control the size of the coefficients, note that Lemma \ref{binomial} implies that if $k\ge 1$ and $|z|\le kT$ and $i\le L$ then
\[
\left| \binom{z}{i} \right| \le (ek)^L (M+4)^L.
\]
And for $j\le M, x\le kT$ we have 
\[
|f(x)|^j\le c_1^M e^{M\delta(kT)}.
\]
Taking $k=1$ here, these estimates together with Siegel's Lemma show that there are integers $p_{i,j}$ not all zero such that \eqref{primes:constraint} holds for all $n\in\A \cap [\epsilon T,T]$ and such that
\begin{equation}\label{primes:coefficients_of_P}
|p_{i,j}| \le T e^L(M+4)^L c_1^M e^{M\delta (T)}.
\end{equation}
Combining this with Lemma \ref{polydiff}, and using \eqref{primes:g_close_to_f} we have
\[
|P(x,f(x))-P(x,g(x))| \le T^3 (ek)^{2L}(M+4)^{2L} c_2^{3M} e^{3M\delta(kT)} e^{-\eta\epsilon T}
\]
where $k\ge 1$ and $x\in [\epsilon T,kT]$. So for $n\in \A\cap [\epsilon T,T]$ we have
\begin{equation}\label{prime:P(n,g(n))}
|P(n,g(n))| \le T^3 e^{2L}(M+4)^{2L} c_2^{3M} e^{3M\delta(T)} e^{-\eta\epsilon T}
\end{equation}
and for $x\in (T,kT]$ we have
\begin{equation}\label{prime:to_beat}
|P(x,f(x))|\le |P(x,g(x))|+ T^3 (ek)^{2L}(M+4)^{2L} c_2^{3M} e^{3M\delta(kT)} e^{-\eta\epsilon T}.
\end{equation}
In order to apply Proposition \ref{jensen} we estimate $|P(z,g(z))|$ for $z$ in the closed right half-plane, with $|z|\le kT$. We have
\begin{equation}\label{prime:modphi}
|P(z,g(z))| \le T^2 (ek)^{2L} (M+4)^{2L} c_2^{2M} e^{2M\delta(kT)}.
\end{equation}

Now let $T_1=\min \{ n\in \A : n>T\}$, and fix $\ell$ such that $T<T_1\le \ell T$ (for instance we can take $\ell$ around $1/\epsilon$). Put $\phi(z)= P(z+T_1,g(z+T_1))$, analytic on $|z|\le T_1$. By \eqref{prime:modphi} with $k=2\ell$ we have
\begin{equation}\label{prime:modphiT_1}
|\phi|_{T_1} \le T^2 (2e\ell)^{2L} (M+4)^{2L} c_2^{2M} e^{2M\delta(2\ell T)}.
\end{equation}
And 
\begin{eqnarray*}
\prod_{n\in\A \cap [\epsilon T,T]} \frac{ |n -T_1|}{T_1} &\le & \left( 1 -\epsilon \frac{T}{T_1}\right)^{ \left[ a\frac{T}{(\log T)^\lambda}\right]} \\
&\le & r^{ \frac{T}{(\log T)^\lambda}}
\end{eqnarray*}
where $r= (1-\epsilon^2)^{a/2}<1$.

Arguing as in the previous proof, for $n\in \A\cap [\epsilon T,T]$ we have
\[
\prod_{m\in\A\cap [\epsilon T,T]\setminus \{n\}} \left| \frac{m-T_1}{m-n}\right| \le  \left(c \log T\right)^{ b \frac{T}{(\log T)^\lambda}},
\]
for some $c>1$ (depending on $\epsilon,\ell$ and $a$ but not $T$ or $n$). 

And as in the previous proof, for $n\in \A\cap [\epsilon T,T]$ we have
\[
\prod_{m\in \A\cap [\epsilon T,T]} \abs{\frac{T_1^2-(m-T_1)(n-T_1)}{T_1^2}} < 1.
\]

Applying Proposition \ref{jensen} and using (\ref{prime:modphiT_1}),(\ref{prime:P(n,g(n))}), and (\ref{prime:to_beat}) and the previous three inequalities, we have
\[
\begin{split}
|P(T_1,f(T_1))| & \le T^4 (2e\ell)^{2L}(M+4)^{2L}c_2^{3M}e^{3M\delta(2\ell T)}r^{\frac{T}{(\log T)^\lambda}}\\
&\cdot\left( 1+ \exp\left(  b \frac{T}{(\log T)^\lambda}\log( c\log T) - \left(\eta\epsilon T+\frac{T}{(\log T)^\lambda}\log r\right)\right)\right)
\end{split}
\]
We consider the two factors on the right separately. First, using the fact that $T=(L+1)(M+1)$, the definition of $\delta(x)$, and taking $M+1$ to be an integer around $(\log (L+1))^{\lambda+1}$, we have
\begin{eqnarray*}
T^4 (2e\ell)^{2L}(M+4)^{2L}c_2^{3M}e^{3M\delta(2\ell T)}r^{\frac{T}{(\log T)^\lambda}}\le T^4 (2e\ell)^{2L}(M+4)^{2L}c_2^{3M}e^{6\ell \frac{M(M+1)(L+1)}{(\log(L+1))^{2\lambda+2}}}r^{\frac{(M+1)(L+1)}{2^\lambda(\log (L+1))^\lambda}}\\
\le \exp\left( (L+1) \left( 6 \log (M+4)+\frac{1}{2^\lambda}(\log r)\log (L+1)\right)\right) 
\end{eqnarray*}
for sufficiently large $L$. Since $r<1$, this is at most $1/4$ for sufficiently large $L$. 

For the second factor we have
\[
 1+ \exp\left( b \frac{T}{(\log T)^\lambda}\log (c\log T) - \left(\eta\epsilon T+\frac{T}{(\log T)^\lambda}\log r\right)\right)<\frac{3}{2}
 \]
 for large enough $L$. So $|P(T_1,f(T_1))|$ is an integer less than $1$, hence $P(T_1,f(T_1))=0$.

We now inductively assume that $T'\ge T_1$ and that $P(n,f(n))=0$ for all $n \in \A\cap [\epsilon T, T']$. We write $T'=(M'+1)(L+1)$. Suppose that $k\ge 1$. For $|z| \le k T'$ and $i\le L$  we have
\begin{equation}\label{prime:ind_binom}
\abs{\binom{z}{i}} \le (3 k e)^L (M'+1)^L.
\end{equation}
Using the fact that $\max \{ |f(x)|^j,|g(x)|^j \} \le c_2^M e^{M\delta (k T')}$ for $x\le k T' $ and $j\le M$, together with \eqref{primes:coefficients_of_P},\eqref{primes:g_close_to_f} and Lemma \ref{polydiff}, this implies that
\begin{equation}\label{prime:ind_Pdiff}
|P(x,f(x))-P(x,g(x))| \le T^3 (3 k e)^{2L} (M+4)^L (M'+1)^L c_2^{3M} e^{3M \delta (k T')} e^{-\eta\epsilon T'}
\end{equation}
for $x\in [\epsilon T',kT']$. In particular for $x \in (T',k T']$ we have 
\begin{equation}\label{prime:ind_to_beat}
|P(x,f(x))| \le |P(x,g(x))| +T^3 (3 k e)^{2L} (M+4)^L (M'+1)^L c_2^{3M} e^{3M \delta (k T')} e^{-\eta\epsilon T'}
\end{equation}
As in the base case we now let $T_1'=\min \{ n\in \A : n>T'\}$ and put $\psi(z)= P(z+T_1',g(z+T_1'))$, analytic for $|z|\le T_1'$.Let $\ell$ such that $T'_1\le \ell T'$. Using Proposition \ref{jensen}, and estimating as in the base case, and using \eqref{prime:ind_to_beat} we have
\[
\begin{split}
|P(T_1',f(T_1'))|\le T^2 (6\ell k)^{2L}(M+4)^L(M'+1)^L c_2^{2M} e^{2M\delta(2 l T')}r^{\frac{T'}{(\log T')^\lambda}} \\+ 2 \frac{T'}{(\log T')^\lambda} T^3 (3\ell e)^{2L}(M+4)^L(M'+1)^L c_2^{3M}e^{3 M \delta (\ell T')}e^{-\eta\epsilon T'} \left(  \log T'\right)^{c_6 \frac{T'}{(\log T')^\lambda}}
\end{split}
\]
for some $r<1$ and some positive $c_6$.  With the choice of $M$ made earlier, we see that for $M'>M$ we have $|P(T_1',f(T_1'))|<1$ once $L$ is fixed sufficiently large. So $P(T_1',f(T_1'))=0$. And then inductively we have $P(n,f(n))=0$ for all sufficiently large $n\in\A$. 

By o-minimality and analyticity it follows that $P(x,f(x))=0$ for all $x\ge 0$. So $f$ is algebraic. It then follows from the theorem on page 131 of \cite{Serre} that $f$ must in fact be a polynomial.

\end{proof}

\bibliographystyle{amsplain}
\bibliography{refs}

\end{sffamily}

\end{document}